\documentclass[11pt,a4paper]{article}

\usepackage{fullpage,graphicx,latexsym,amsfonts,amssymb,amsmath,amsthm,relsize}
\usepackage[hyperfootnotes=false]{hyperref}
\usepackage[normalem]{ulem} %
\usepackage{cancel} %
\newtheorem{theorem}{Theorem}
\newtheorem{lemma}{Lemma}
\newtheorem{corollary}{Corollary}
\newtheorem{observation}{Observation}
\newtheorem{proposition}{Proposition}
\newtheorem*{claim}{Claim}
\usepackage[usenames,dvipsnames]{xcolor}

\newcommand{\R}{\ensuremath{\mathbb R}}
\newcommand\blfootnote[1]{%
  \begingroup
  \renewcommand\thefootnote{}\footnote{#1}%
  \addtocounter{footnote}{-1}%
  \endgroup
}

\newcommand{\Xkl}[1][k,\ell]{\ensuremath{X_{#1}}}
\newcommand{\Xk}[1][k]{\ensuremath{X_{#1}}}
\newcommand{\ASum}[1][\ell]{\ensuremath{A_{#1}}}

\newcommand{\mrk}[2][r]{\ensuremath{m_{#1}(#2)}}

\newcommand{\fkl}[2]{\ensuremath{f(#1,#2)}}
\newcommand{\FSum}{\ensuremath{F}}

\newcommand{\FrSum}[1][r]{\ensuremath{F_{#1}}}

\DeclareMathOperator{\Fib}{Fib}

\graphicspath{{figures/}}

\title{%
	On weighted sums of numbers of convex polygons in point sets}
\author{Clemens Huemer\thanks{Departament de Matem{\`a}tiques,
		Universitat Polit{\`e}cnica de Catalunya, 
		{\tt clemens.huemer@upc.edu}}
	\and
	  	Deborah Oliveros\thanks{Instituto de Matem\'aticas, 
		Universidad Nacional Aut\'onoma de M\'exico,
		{\tt dolivero@matem.unam.mx}}		
	\and
	Pablo P\'erez-Lantero\thanks{Departamento de Matem\'atica y Ciencia de la Computaci\'on, USACH, Chile. {\tt pablo.perez.l@usach.cl}.}
		\and
		Ferran Torra\thanks{Departament de Matem{\`a}tiques,
		Universitat Polit{\`e}cnica de Catalunya, 
		{\tt ferran.torra@estudiant.upc.edu }}
	\and
		Birgit Vogtenhuber\thanks{Institute of Software Technology,
        	Graz University of Technology, Austria,
        	{\tt bvogt@ist.tugraz.at}}
}

\date{\today} 

\begin{document}
\maketitle
\begin{abstract}
Let $S$ be a set of $n$ points in general position in the plane, and let $X_{k,\ell}(S)$ be the number of convex $k$-gons with vertices in $S$ that have exactly~$\ell$ points of $S$ in their interior. We prove several equalities for the numbers $X_{k,\ell}(S)$. This problem is related to the Erd\H{o}s-Szekeres theorem. Some of the obtained equations also extend known equations for the numbers of empty convex polygons to polygons with interior points. Analogous results for higher dimension are shown as well.
\end{abstract}

\section{Introduction}\label{intro}
\blfootnote{
		\begin{minipage}[l]{0.3\textwidth} 
		\vspace{-0.1cm}
		\hspace{-0.25cm}
		\includegraphics[trim=10cm 6cm 10cm 5cm,clip,scale=0.11]{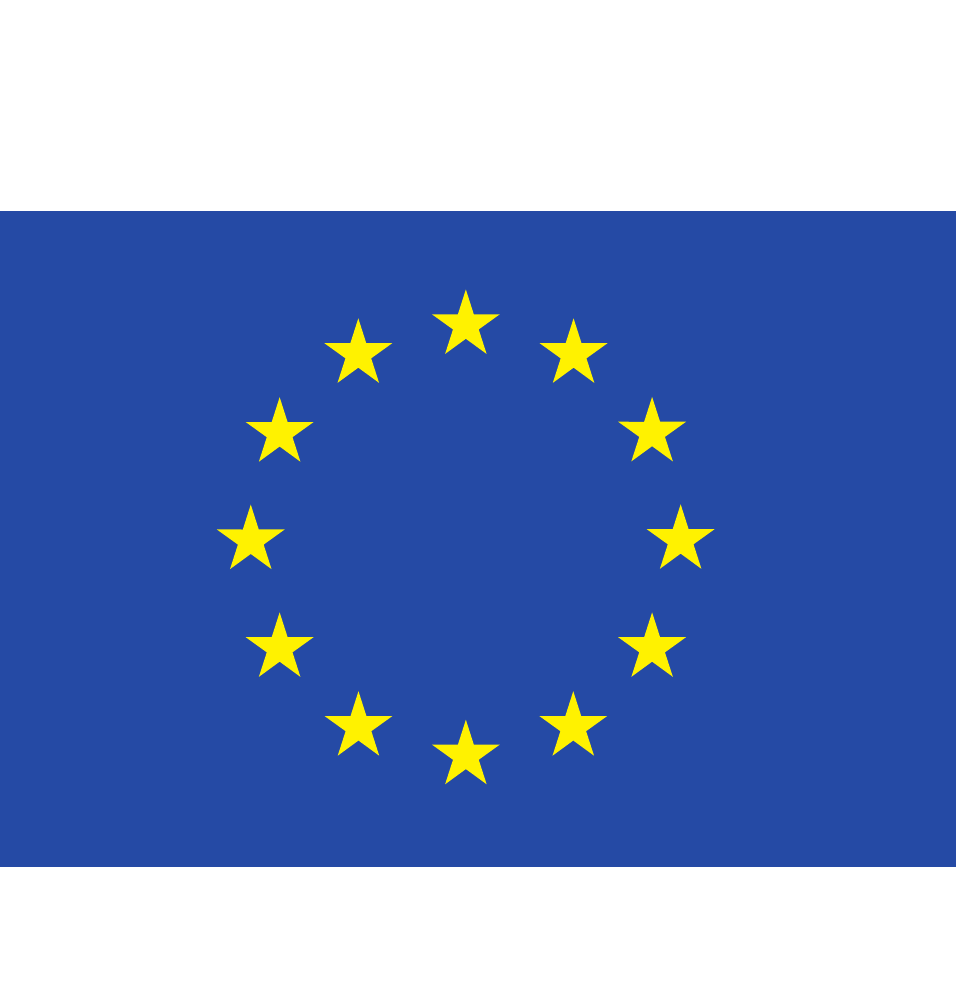} 
	\end{minipage}  
	\hspace{-3.6cm} 
	\begin{minipage}[l][1cm]{0.87\textwidth}
		\vspace{0.21cm}
		This project has been supported by the European Union's Horizon 2020 research and innovation programme under the Marie Sk\l{}odowska-Curie grant agreement No 734922.
 	\end{minipage}

\vspace{-0.01cm}
\hspace{0.0cm}
\begin{minipage}[l]{0.97\textwidth}
\noindent Research of C.H. was partially supported by project MTM2015-63791-R (MINECO/FEDER), and by project Gen. Cat. DGR 2017SGR1336.
Research of D.O. was partially supported by project PAPIIT IN106318 and IN107218 as well as CONACyT 282280.
Research of P.~P-L. was partially supported by project DICYT 041933PL Vicerrector\'ia de Investigaci\'on, Desarrollo e Innovaci\'on USACH (Chile), 
and Programa Regional STICAMSUD  19-STIC-02. 
Research of B.V. was partially supported by the Austrian Science Fund (FWF) within the collaborative DACH project \emph{Arrangements and Drawings} as FWF project \mbox{I 3340-N35}. 
\end{minipage}
}

Let $S$ be a set of $n$ points in the Euclidean plane in general position, that is, no three points of $S$ are collinear.
We denote by $h$ the number of extreme points of $S$, that is, the points of~$S$ that lie on the boundary of the convex hull of $S$. 
A $k$-gon (of $S$) is a simple polygon that is spanned by exactly $k$ points of $S$. 
A $k$-gon is called \emph{empty}, if it does not contain any points of $S$ in its interior. The vertex set of an empty convex $k$-gon is sometimes also called a {\it{free}} set~\cite{EdelmanR00}.
We denote by $\Xk(S)$ the number of empty convex $k$-gons in $S$, and more general, we denote by $\Xkl(S)$ the number of convex $k$-gons in $S$ that have exactly $\ell$ points of $S$ in their interior.  
Further, a convex $k$-gon with $\ell$ interior points constitutes a subset of $(k+\ell)$ points of~$S$ whose convex hull does not contain any other points of $S$. Such sets are sometimes also called {\it{islands}}~\cite{islas11}. Figure~\ref{fig:smallexamples5_2} shows a set $S$ of $10$ points and its values $\Xkl(S).$  Here we are interested in relations among the values $\Xkl(S)$ and invariants among all point sets of given cardinality. 
\begin{figure}[htb]
\centering
\includegraphics[page=1,scale=1]{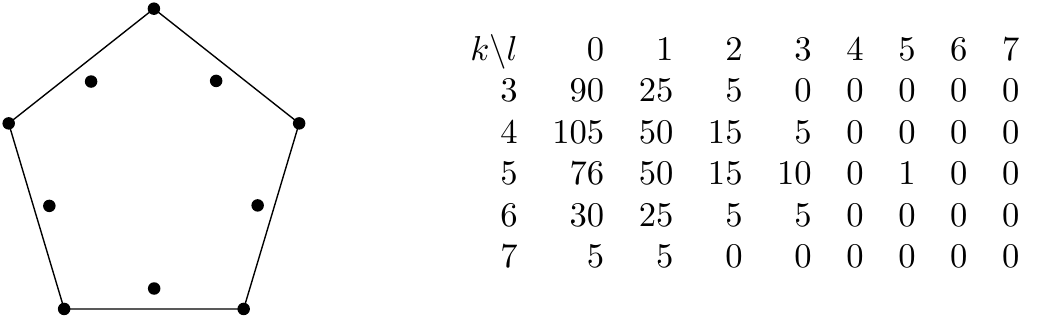} %
\caption{A set $S$ of $n=10$ points, with $h=5$ of them on the boundary of the convex hull, and the numbers of convex $k$-gons with $\ell$ interior points, $X_{k,\ell}(S)$.}
\label{fig:smallexamples5_2}
\end{figure}
Invariants for the values $\Xk(S)$ are well known. The equation 
\begin{equation}
\label{eqn_edelman}
\sum\limits_{k\geq 3} (-1)^{k+1} \Xk(S)  = \ \binom{n}{2}-n+1
\end{equation} was proved by Edelman and Jamison in~\cite{EdelmanJ85}, actually in terms of the number $F_k$ of free sets of cardinality $k$ in a convex geometry (Theorem 4.5 in~\cite{EdelmanJ85}), stating $\sum_{k}(-1)^{k}F_k=0.$  Note that in our notation we omit the term for the empty set $X_0:=1$, and $X_1(S)=n$ and $X_2(S)={{n}\choose{2}}$, the number of points and edges spanned by $S$, respectively. In~\cite{EdelmanJ85}, Equation~(\ref{eqn_edelman}) is also attributed to J.~Lawrence. Moreover, it also holds in higher dimension.  
The equation 
\begin{equation}
\label{eqn_ahrens}
\sum\limits_{k\geq 3} (-1)^{k+1} k \Xk(S)  = \ 2\binom{n}{2}-h,
\end{equation} which only depends on $n$ and on $h$, was first proved by Ahrens, Gordon, and McMahon~\cite{AhrensGM99}. Its higher dimensional version was proved by Edelman and Reiner~\cite{EdelmanR00} and by Klain~\cite{Klain99}.
Pinchasi, Radoi\v{c}i\'c, and Sharir~\cite{prs_oecp_06} provided elementary geometric proofs for Equations~(\ref{eqn_edelman}) and~(\ref{eqn_ahrens}), by using a continous motion argument of points.
They further proved other equalities and inequalites involving the values $\Xk(S)$ and extensions to higher dimension. 
As one main result, they showed that for planar point sets and $r \geq 2$, is holds that
\begin{equation}
\label{eqn_tr}
\sum\limits_{k\geq 2r} (-1)^{k+1} \frac{k}{r} \binom{k-r-1}{r-1} \Xk(S)  = \ -T_r(S)
\end{equation}
where $T_r(S)$ %
denotes the number of $r$-tuples of vertex-disjoint edges $(e_1,\dots , e_r)$ spanned by $S$
that lie in convex position, and are such that the region $\tau(e_1,\dots , e_r)$, formed by the intersection of the $r$ half-planes that are bounded by the lines supporting $(e_1,\dots , e_r)$ and contain the other edges, has no point of $S$ in its interior~\cite{prs_oecp_06}. 
For the point set $S$ of Figure~\ref{fig:smallexamples5_2} we have $T_2(S)=30$, $T_3(S)=25$, and $T_r(S)=0$ for $r\geq 4.$ 

In~\cite{prs_oecp_06}, the alternating sums of Equations (\ref{eqn_edelman}), (\ref{eqn_ahrens}) and (\ref{eqn_tr}) have been joined by the notion of the \emph{$r$-th alternating moment} $M_r(S)$ of $\{\Xk(S)\}_{k\geq 3}$ as
$$
	M_r(S) := \left\{ \begin{array}{ll}
		\sum\limits_{k\geq 3} (-1)^{k+1} \Xk(S) & \mbox{ for } r = 0 \\
		\sum\limits_{k\geq 3} (-1)^{k+1} \frac{k}{r} \binom{k-r-1}{r-1} \Xk(S) & \mbox{ for } r \geq 1. \  \footnotemark
	\end{array} \right.
$$
\footnotetext{Note that Equation~(\ref{eqn_tr}) and its proof require that $k\geq 2r$.}
In this work, we take a similar approach as done in the work of Pinchasi et al.~\cite{prs_oecp_06} and extend above-mentioned results to $\Xkl(S)$, that is, to convex $k$-gons having a fixed number $\ell$ of points of $S$ in their interior. 
We denote by
\[	\ASum(S) := \sum\limits_{k\geq 3} (-1)^{k+1} \Xkl(S) \]
the \emph{alternating sum} of $\{\Xkl(S)\}_{k\geq 3}$. In Section~\ref{sec:alt_sums}, we mostly concentrate on the case of polygons with one interior point, i.e., $\ell=1$, where we prove that  
$$ \ASum[1](S)  =  \sum\limits_{k\geq 3} (-1)^{k+1} X_{k,1}(S) = n-h $$ for any set $S$ of $n$ points with $h$ of them on the boundary of the convex hull. 
Further, equations for alternating sums of convex polygons with one given interior point $p \in S$ (denoted $A_1^p(S)$) and for alternating sums of convex polygons with one interior point and which contain a given edge $e$ spanned by two points of $S$ on the boundary (denoted $A_1(S;e)$) are obtained. 
We also derive inequalities for $\sum_{k=3}^{t}(-1)^{k+1}X_{k,1}(S)$ for any given value $t\geq 3$, based on analogous inequalities for $X_{k}$ from~\cite{prs_oecp_06}.
In Section~\ref{sec:weighted_sums}, we consider different weight functions $\fkl{k}{\ell}$ and general weighted sums over $\Xkl(S)$ of the form 
$$ \FSum(S) = \sum_{k\geq 3} \sum_{\ell\geq 0} \fkl{k}{\ell} \Xkl(S).$$
We show that for any function $\fkl{k}{\ell}$ with the property that 
\begin{equation}
\label{eq:rr_general1}
\fkl{k}{\ell} = \fkl{k+1}{\ell-1} + \fkl{k}{\ell-1},
\end{equation} 
the value of the according sum $\FSum(S)$ is invariant over all sets $S$ of $n$ points.
We present several functions $\fkl{k}{\ell}$ with this property. Among these functions, we find that
for any point set $S$ of $n$ points in general position and for any $x \in \R$, it holds that
$$
\sum_{k=3}^n \sum_{\ell=0}^{n-k} x^k \left(1+x\right)^\ell X_{k,\ell} = \left(1+x\right)^n -1 - x\cdot n - x^2 \binom{n}{2}.
$$
Note that Equation~(\ref{eqn_edelman}) is the special case $x=-1$; for this case and $\ell=0$, set the indeterminate form $\left(1+x\right)^\ell:=1$.
We further show that the maximum number of linearly independent equations 
$F_j(S) =\sum_{k \geq 3} \sum_{\ell \geq 0} f_j(k,\ell) X_{k,\ell}$, where each $f_j(k,\ell)$ satisfies Equation (\ref{eq:rr_general1}), in terms of the variables $X_{k,\ell}$, is $n-2$.
In Section~\ref{sec:moment_sums}, we relate the results of Section~\ref{sec:weighted_sums} to the moments from~\cite{prs_oecp_06}.
We denote by $\mrk{k}$ the multiplicative factor of $(-1)^{k+1}X_k(S)$ in $M_r(S)$, that is, 
\[
	\mrk{k} := \left\{\begin{array}{ll} 
				1 & \mbox{for\ } r=0 \\ 
				\frac{k}{r}\binom{k-r-1}{r-1} & \mbox{for\ } r \geq 1 \ \mbox{and $k \geq 2r$}\\
				0 & \mbox{otherwise}
				\end{array} \right.
\] 
We show the following relations, which only depend on the number $n$ of points. For any point set $S$ with cardinality $n$ and integer $0 \leq r\leq 2$ it holds that
\[ 
	\FrSum(S) := \sum_{k\geq 3} \sum_{\ell=0}^r (-1)^{k-\ell+1} \mrk[r-\ell]{k-\ell} \Xkl(S) = 
		\left\{\begin{array}{ll}
				\binom{n}{2}-n+1 		& \mbox{for\ } r=0 \\
				2\binom{n}{2}-n 		& \mbox{for\ } r=1 \\
				-\binom{n}{2}+n				& \mbox{for\ } r=2 \\ 
		\end{array}\right.
\]
Finally, in Section~\ref{sec:higher_dimensions} we discuss the generalization of the obtained results to higher dimensions.
Several more results, in particular on identities involving $X_{k,\ell}$ for special point configurations, can be found in the thesis~\cite{Torra19}.\\

An important argument that will be used in the proofs of this work is the continuous motion argument of points. 
 Of course, this argument is not new for the analysis of configurations in combinatorial geometry (see for example~\cite{AichholzerGOR07,AndrzejakAHSW98,prs_oecp_06,Tverberg66}). 
 The goal is to prove a property for all point sets in general position in the plane. To this end, the property is first shown to hold for some particular point set (usually a set of points in convex position). 
 Then, one can move the points from the point set, such that only one point is moved at each instant of time, until any particular point configuration is reached. 
 It remains to show that the property holds throughout. The changes on the combinatorial structure of the point set only appear when during a point move, one point becomes collinear with two other points. 
 It hence is sufficient to check that the property is maintained if one point crosses the edge spanned by two other points of the set. 
 An analogous proof strategy can be applied in $\R^d$ for $d>2$. 

\subsection{Related work}

The problem studied in this work is related to one of the most famous problems in combinatorial geometry, namely the one of showing that every set of sufficiently many points in the plane in general position determines a convex $k$-gon. 
The problem of determining the smallest integer~$f(k)$ such that every set of $f(k)$ points contains a convex $k$-gon was original inspiration of Esther Klein and 
has become well-known as the Erd\H{o}s-Szekeres problem.  
The fact that this number $f(k)$ exists for every $k$ and therefore, for given $k$, $\sum_{\ell \geq 0} \Xkl(S)\geq 1$  when  $S$ has sufficiently many points, was first established in a seminal paper of Erd\H{o}s and Szekeres~\cite{ES_1935} who proved the following bounds on $f(k)$~\cite{ES_1935,ES_1960}:
\[ 2^{k-2} + 1 \le f(k) \le \binom{2k-4}{k-2} + 1 .\] 
Subsequently, many improvements have been presented on the upper bound; 
see~\cite{Morris} for a survey on this problem.
Very recently, the problem has been almost settled by a work of Suk~\cite{Suk}, who showed that $f(k)=2^{k+o(k)}$. 
Another recent related work concerning the existence of convex $k$-gons in point sets is~\cite{DuqueMH18}. 
In a slight variation of the original problem, 
Erd\H{o}s suggested to find the minimum number of points $g(k)$ in the plane in general position containing $k$ points which form an empty convex $k$-gon. 
It is easy to show that for empty triangles and empty convex quadrilaterals this number is $3$ and $5$, respectively. 
Harborth~\cite{Harborth} showed in 1978 that $g(5)=10$. 
Thus $X_{k,0}(S)\geq 1$ for $k=3,4,5$ when $S$ has at least $g(k)$ points. 
However, in 1983 Horton~\cite{Horton} constructed an infinite set of points with no empty convex $7$-gon, implying that $g(k)$ is infinite for any $k \geq 7$.
Much later, Overmars \cite{Overmars} used a computer to find a set with $29$ points without an empty convex hexagon and later Gerken \cite{Gerken} and Nicolas~\cite{Nicolas} proved that $g(6)$ is finite. 

A weaker restriction of the convex polygon problem has been considered by Bialostocki, Dierker and Voxman \cite{Bialostocki-etal}. 
They conjectured that for any two integers $k\geq 3$ and $\ell \geq 1$, there exists a function $C(k,\ell)$ such that every set with at least $C(k,\ell)$ points contains a convex $k$-gon whose number of interior points is divisible by $\ell$. They also showed that their conjecture is true if $k \equiv 2 (\operatorname{mod} \ell)$ or $k \geq \ell+3$.

In parallel %
to questions concerning the existence of certain configurations in every large enough set of points in general position, also their number has been subject of research.
For example, Fabila-Monroy and Huemer~\cite{RuyClemensIslands12} considered the number of $k$-islands in point sets.
They showed that
for any fixed $k$, 
their number is in $\Omega(n^2)$ for any $n$-point set in general position 
and in $O(n^2)$ for some such sets.%

The question of determining the number of convex $k$-gons contained 
in any $n$-point set in general position was raised in the 1970s
by Erd\H{o}s and Guy~\cite{EG_1973}. 
The trivial solution for the case $k=3$ is $\binom{n}{3}$.  
However, already for convex \mbox{4-gons} this question turns out to be highly non-trivial, 
as it is related to the search for 
the minimum number of crossings in a straight-line drawing 
of the complete graph with $n$ vertices; see again~\cite{EG_1973}.

Erd\H{o}s also %
posed the respective question for the number $h_k(n)$ of empty convex $k$-gons~\cite{ER84}.
Horton's construction implies $h_k(n) = 0$ for every $n$ and every $k \geq 7$,
so it remains to consider the cases $k\!=\!3, \ldots, 6$.
For the functions $h_3(n)$ and $h_4(n)$, asymptotically tight estimates are known.
The currently best known bounds are
$n^2 + \Omega(n\log^{2/3}n) \le h_3(n) \le  1.6196n^2+o(n^2)$ 
and $\frac{n^2}{2} + \Omega(n\log^{3/4}n) \le h_4(n) \le 1.9397n^2+o(n^2)$,
where the lower bounds can be found in~\cite{Many5Holes_ARXIV} and the upper bounds 
are due to B\'{a}r\'{a}ny and Valtr~\cite{BV2004}.
For $h_5(n)$ and $h_6(n)$, no matching bounds are known.
The best known upper bounds $h_5(n)\le 1.0207n^2+o(n^2)$ and $h_6(n) \leq 0.2006n^2+o(n^2)$
can also be found in~\cite{BV2004}.
The best known lower bound $h_6(n) \ge n/229 - 4$ is due to Valtr~\cite{v-ephpp-12}. 
It is widely conjectured that $h_5(n)$ grows quadratically in~$n$. 
However, despite many efforts in the last 30 years, only very recently a superlinear bound 
of $h_5(n) = \Omega(n\log^{4/5}{n})$ has been shown~\cite{abhkpsvv-slbnh-17}.
A result of independent interest is by Pinchasi, Radoi\v{c}i\'c, and Sharir~\cite{prs_oecp_06}, 
who showed $h_4(n) \geq h_3(n) - \frac{n^2}{2} - O(n)$ and $h_5(n) \geq h_3(n) - n^2 - O(n)$. 
By this, any improvement of the constant 1 for the dominating factor $n^2$ of the lower bound of $h_3(n)$ 
would imply a quadratic lower bound for $h_5(n)$.

%
%
%
%
%
%

\section{Alternating sums}\label{sec:alt_sums} %
In this section, we concentrate on alternating sums of numbers of polygons with one interior point and possibly with some elements fixed. 
To this end, we introduce some more notation. 
Let $p$, $q$, and $r$ be three points of a set $S$ of $n$ points in general position in the plane. We denote by $\Delta{pqr}$ the triangle with vertices $p$, $q$, and
$r$, and by a directed edge $e=pq$ the segment that connects $p$ and $q$ and is oriented from $p$ to $q$. We do not always specify if an edge is directed or undirected when it is clear from the context. We sometimes also write polygon instead of convex polygon since all considered polygons are convex.  
We say that a polygon lies to left side of a directed edge $e$ if it contained in the left closed half-plane that is bounded by the line through~$e$. 
For a fixed point $p\in S$, we denote by $\Xkl[k,1]^p(S)$ the number of convex $k$-gons spanned by $S$ that contain exactly $p$ in their interior, and by 
$\ASum[1]^p(S):=  \sum_{k\geq 3} (-1)^{k+1} \Xkl[k,1]^p(S)$ the according alternating sum.
Further, for a directed edge $e$, we denote by $\Xkl(S;e)$ the number of convex $k$-gons with $\ell$ interior points that have $e$ as a boundary edge and lie on the left side of $e$; 
and by $\ASum(S;e):=  \sum_{k\geq 3} (-1)^{k+1} \Xkl(S;e)$ the according alternating sum. 
Likewise $\Xkl(S;p)$ is the number of convex $k$-gons with $\ell$ interior points that have $p$ on their boundary.
If more elements (points and/or edges) are required to be on the boundary of the polygons, then all those are listed after the semicolon in this notation. 
Further, if an element is required to not be on the boundary of the polygons, then it is listed with a minus.
For example, $\Xkl(S;e,p_1,-p_2)$ denotes the number of convex $k$-gons with $\ell$ interior points that are on the left of $e$, and have $e$ and $p_1$ on their boundary but not $p_2$.

In~\cite{prs_oecp_06}, the authors show as a side result (in the proof of Theorem 2.2) that the alternating sum 
of convex $k$-gons incident to (the left side of) a directed edge $e=pq$ is 1 if there is at least one point of $S\backslash\{p,q\}$ in the left side of $e$, and 0 otherwise. 
As we will repeatedly use this result, we explicitly state it as a lemma here. 
We remark that in~\cite{prs_oecp_06}, $\ASum[0](S)$ and $\ASum[0](S;e)$ are denoted as $M_0(S)$ and $M_0(e)$, respectively. %

\begin{lemma}[\cite{prs_oecp_06}] 
	\label{edgelemma}
	For any set $S$ of $n$ points in general position in the plane and any directed edge $e=pq$ spanned by two points of $S$, 
	it holds that
	\[ 
		\ASum[0](S;e) = \left\{\begin{array}{ll}
					1 & \mbox{ if $e$ has at least one point of $S\backslash\{p,q\}$ to its left} \\
					0 & \mbox{ otherwise.} \\
					\end{array}\right.
	\]

\end{lemma}

\begin{lemma}\label{lem:one_fixed_point}
	For any set $S$ of $n$ points in general position in the plane and any point $p \in S$ it holds that 
	\[ 
		\ASum[1]^p(S) = \left\{\begin{array}{ll}
					0 & \mbox{ if $p$ is an extreme point of $S$} \\
					1 & \mbox{ otherwise.} \\
					\end{array}\right.
	\]
\end{lemma}
\begin{proof}
	Obviously, if $p$ is an extreme point of $S$, then it cannot be in the interior of any polygon spanned by points of $S$ and hence $\ASum[1]^p(S) = 0$.
	So assume that $p$ is not an extreme point of~$S$. Note that every polygon in $S$ that contains exactly $p$ in its interior is an empty polygon in $S\setminus\{p\}$.

	If $p$ lies close enough to a convex hull edge $e$ of $S$ (an edge on the boundary of the convex hull of $S$), then $p$ is contained in exactly all polygons that are incident to $e$ and empty in $S\setminus\{p\}$.
	Hence, $\ASum[1]^p(S) = \ASum[0](S\setminus\{p\};e) = 1$ by Lemma~\ref{edgelemma}. %
	Otherwise, if $p$ is located arbitrarily, consider a continuous path
	from $p$ to a position close enough to a convex hull edge of $S$, and move $p$ along this path. This path can be chosen such that it avoids all crossings in the line arrangement spanned by $S\setminus\{p\}$ and lies inside the convex hull of $S$.  
	During this movement, $\ASum[1]^p(S)$ can only change when $p$ crosses an edge $qr$ spanned by two points of $S$. 
	Further, changes can only occur from changing amounts of convex $k$-gons that have $qr$ as an edge. 
	More exactly, when $p$ is moved over~$qr$, from its left to its right side, then the alternating sum of polygons that $p$ ``stops being inside'' is $\ASum[0](S\setminus\{p\};qr) = 1$, and the alternating sum of polygons that~$p$ ``starts being inside'' is $\ASum[0](S\setminus\{p\};rq) = 1$ (note that $qr$ is not a convex hull edge of $S$). 
	Hence, the value $\ASum[1]^p(S)$ is the same for all possible positions of $p$ on the path, including the final position for which we already showed $\ASum[1]^p(S) = 1$. %
\end{proof}

\begin{theorem}\label{thm:one_point}
	Given a set $S$ of $n$ points in general position in the plane, $h$ of them extreme, it holds that $\ASum[1](S) = n-h$.
\end{theorem}
\begin{proof}
	Any convex $k$-gon counted	in $\ASum[1](S)$ contains exactly one point in its interior. 
	Hence, we can count the convex $k$-gons by their interior points and obtain that $\ASum[1](S)=\sum_{p\in S} \ASum[1]^p(S) = n-h$.
\end{proof}

In the previous result we used $\ASum[0](S;e)$ to obtain bounds on $\ASum[1]^p(S)$ and determine $\ASum[1](S)$. %
A possible approach for determining $\ASum[2](S)$ could be via $\ASum[1](S;e)$. In the following, we show why such an approach cannot work.

\begin{lemma}\label{one_fixed_edge_one_fixed_point}
	Given a set $S$ of $n$ points in general position in the plane, a point $p \in S$, and a directed edge $e=qr$ of $S\setminus\{p\}$, it holds that
	$\ASum[1]^p(S;e) \in \{ -1, 0, 1 \}$.
\end{lemma}
\begin{proof}
	The idea for obtaining $\ASum[1]^p(S;e)$ is to start with $\ASum[0](S\setminus \{p\};e) = 1$ 
	and subtract from it the alternating sum of all convex polygons that stay empty when adding $p$ again to $S$.  
	In the following, we denote the latter with $\ASum[0](S;e,-p)$. So we have 
	\[ \ASum[1]^p(S;e) = \ASum[0](S\setminus \{p\};e) - \ASum[0](S;e,-p). \]
	As we are only counting polygons on the left side of $e$, we assume without loss of generality that $e$ is a convex hull edge of $S$, 
	and $S$ lies to the left of $e$.

	Obviously, if $p$ is an extreme point of $S$, then it cannot be in the interior of any polygon spanned by points of $S$, and hence $\ASum[1]^p(S;e)=0$.
	Likewise, if the triangle $\Delta{pqr}$ contains any points of $S$ in its interior, $\ASum[1]^p(S;e)=0$.
	So assume that $p$ is not an extreme point of $S$ and that $\Delta{pqr}$ is interior-empty. 
	Consider the supporting lines $\ell_q$ and $\ell_r$ of $pq$ and $pr$, respectively, and the four wedges bounded by these lines. 
	One of the wedges contains the triangle $\Delta{pqr}$, one contains $p$ and $q$ but not $r$ (we call it the $q$-wedge), 
	one contains $p$ and $r$ but not $q$ (we call it the $r$-wedge), 
	and the last one contains $p$ but not $q$ and $r$ (we call it the $p$-wedge); see Figure~\ref{fig:wedges}.

\begin{figure}[htb]
\centering
\includegraphics[page=1,scale=1]{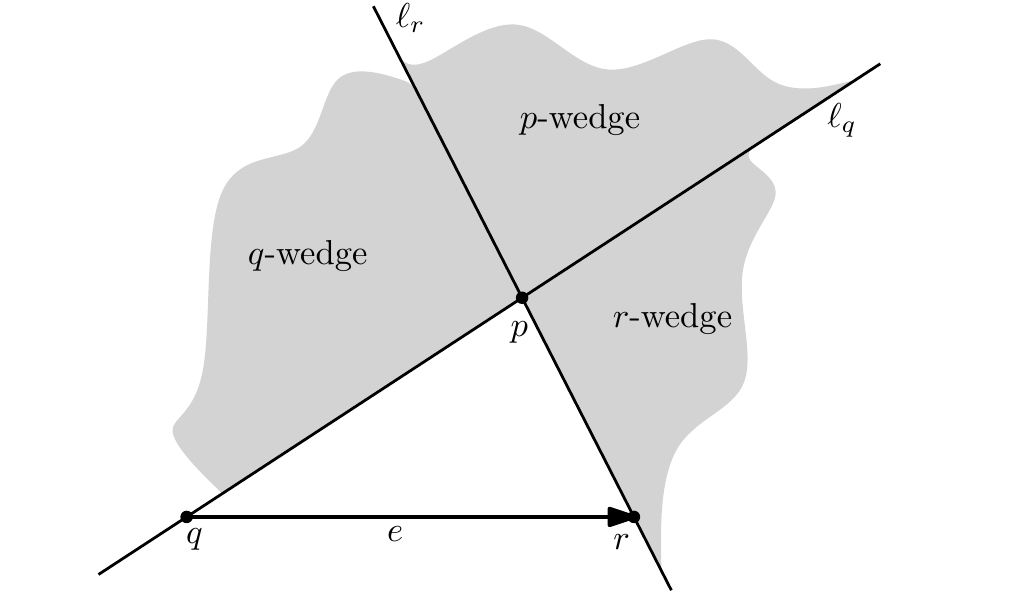} %
\caption{Illustration of wedges.}
\label{fig:wedges}
\end{figure}

	Note that any convex polygon that has $e$ as an edge and a vertex in the $p$-wedge contains $p$. Hence, any polygon counted in $\ASum[0](S;e,-p)$ has all its vertices (except $q$ and $r$) in the interior of the 
	$q$-wedge and the $r$-wedge.
	If both, the $q$-wedge and the $r$-wedge are empty of points (i.e., $p$ lies close enough to $e$) then $\ASum[0](S;e,-p)=0$ and $p$ is contained in exactly all polygons that are incident to $e$ and empty in $S\setminus\{p\}$. 
	Hence, $\ASum[1]^p(S;e) = \ASum[0](S\setminus\{p\};e) = 1$ by Lemma~\ref{edgelemma}. %
	So assume that at least one of the $q$-wedge and the $r$-wedge, without loss of generality\ the $r$-wedge, contains at least one point of $S$ in its interior.

	We first count all polygons for $\ASum[0](S;e,-p)$ that have all vertices except $q$ and $r$ in the interior of the $r$-wedge. 
	The alternating sum of those polygons is $\ASum[0](S \cap H_q;e,-p)=\ASum[0](S \setminus \{p\} \cap H_q;e)$, 
	where $H_q$ is the closed half-plane bounded by $\ell_q$ that contains $e$. 
	By Lemma~\ref{edgelemma}, $\ASum[0](S \cap H_q;e,-p) = 1$. %

	For counting the polygons in $\ASum[0](S;e,-p)$ that have at least one vertex in the $q$-wedge, 
	we consider the points in the interior of the $q$-wedge in counterclockwise order around $p$ 
	(such that $q$ is ``after'' the last point) and denote them by $p_1, \ldots p_m$.
	Then for every $i \in \{1, ... m\}$, consider the line $\ell_i$ through $p$ and $p_i$ and the closed half-plane $H_i$ bounded by $\ell_i$ that contains $qr$. 
	Let $S_i=S\cap H_i$;  %
	see Figure~\ref{fig:qwedgesorting}.

\begin{figure}[htb]
\centering
\includegraphics[page=2,scale=1]{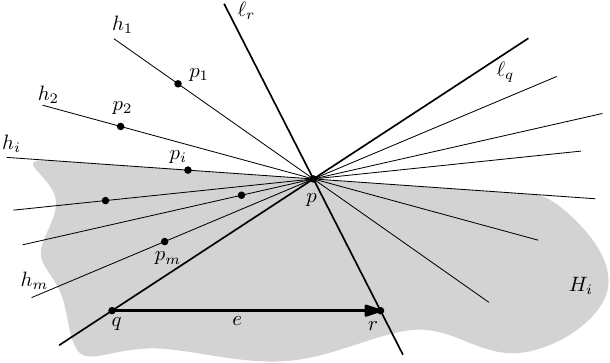} %
\caption{Sorting in $q$-wedge and definition of $H_i$.}
\label{fig:qwedgesorting}
\end{figure}

\begin{claim} \emph{
	The alternating sum $\ASum[0](S_i;e,p_i,-p)$ of empty polygons lying in $H_i$ 
	that contain the triangle $\Delta{p_iqr}$ and do not have $p$ as a vertex is $1$ if $S_i = \{p, q, r, p_i\}$, and $0$ otherwise. 
	}
\end{claim}

	We first complete the proof %
	under the assumption that the claim holds and then prove the claim.
	Note that every counted polygon that has a vertex in the $q$-wedge, has a unique first such vertex in the cyclic order around $p$ for which we can count it.
	Hence, the total alternating sum over those polygons is the sum of $i\in\{1, \ldots, m\}$ of the alternating sums $\ASum[0](S_i;e,p_i,-p)$ with $p_i$ being the first used point in the $q$-wedge.
	Further, note that for any $i<m$ it holds that $p_m \in S_i$ and hence $S_i\neq \{p, q, r, p_i\}$. 
	Thus, the alternating sum of those empty polygons with at least one vertex in the $q$-wedge is 
	$\sum_{i=1}^m \ASum[0](S_i;e,p_i,-p) = \ASum[0](S_m;e,p_m,-p) \in \{0,1\}$. 

	Altogether, we obtain $\ASum[0](S;e,-p) = \ASum[0](S \cap H_q;e,-p) + \ASum[0](S_m;e,p_m,-p) \in \{1,2\}$. %
	Combining this with $\ASum[0](S\setminus \{p\};e) = 1$, we obtain that $\ASum[1]^p(S;e) = \ASum[0](S\setminus \{p\};e) - \ASum[0](S;e,-p) \in\{0,-1\}$ 
	if at  least one of the $q$-wedge and the $r$-wedge contains points of $S$. The lemma thus follows.

	\medskip
	{\it Proof of Claim.}
	If $S_i=\{p, q, r, p_i\}$, then the only such empty polygon in $H_i$ is the triangle $\Delta{p_iqr}$, and hence the alternating sum 
	$\ASum[0](S_i;e,p_i,-p)$ is equal to 1.
	If $S_i\neq\{p, q, r, p_i\}$, then we consider the triangle $\Delta{p_iqr}$, which splits $H_i$ into three wedges; one is bounded by lines $\ell_i$ and $\ell_q$, another one by $\ell_q$ and~$\ell_r$, and the third one by $\ell_r$ and $\ell_i$. We distinguish three cases. 

	\smallskip
	Case 1. The triangle $\Delta{p_iqr}$ contains a point of $S_i$ in its interior. Then no polygon containing $\Delta{p_iqr}$ can be empty. Hence, the alternating 
	sum $\ASum[0](S_i;e,p_i,-p)$ is 0.

	\smallskip
	Case 2. All points of $S_i\setminus\{p_i,p,q,r\}$ lie in one of the other two wedges, without loss of generality the one bounded by $\ell_i$ and $\ell_q$.
	Then, each counted convex $k$-gon in $\ASum[0](S_i;e,p_i,-p)$ with $k\geq 4$ vertices corresponds to a convex $(k-1)$-gon incident to $qp_i$ and not having $r$ as a vertex. 
	The alternating sum of those is $1$ by Lemma~\ref{edgelemma}. %
	Inverting all signs and adding 1 for the triangle $\Delta{p_iqr}$, we obtain a total of $-1+1=0$.

	\smallskip
	Case 3. Both wedges contain points of $S_i\setminus\{p_i,p,q,r\}$ in the interior (and the triangle $\Delta{p_iqr}$ is empty). 
	We have three different types of polygons that we have to count: 
	(i) the triangle $\Delta{p_iqr}$, 
	(ii) the polygons having additional vertices in only one wedge, and 
	(iii) the polygons having additional vertices in both wedges. 
	For the latter, note that the union of the vertex sets of two empty polygons, one contained in each wedge, and containing the edge $qp_i$, respectively $p_ir$, gives a polygon that is counted in $\ASum[0](S_i;e,p_i,-p)$. %

	Let $S_i^q$ and $S_i^r$ be the points of $S_i\setminus\{p\}$ in the wedges bounded by $\ell_i$ and $\ell_q$, and bounded by $\ell_r$ and $\ell_i$, respectively. 
	Note that $\ASum[0](S_i^q;qp_i)= \ASum[0](S_i^r;p_ir) = 1$. 
	Let $L=\sum_{k\geq 3, \ k \mbox{ odd}} \Xkl[k,0](S_i^q;qp_i)$ be the number of such convex $k$-gons with odd $k$. 
	As $\ASum[0](S_i^q;qp_i)=1$, the respective number of even polygons is $L - \ASum[0](S_i^q;qp_i)= L-1$. 
	Similarly, let $R=\sum_{k\geq 3, \ k \mbox{ odd}} \Xkl[k,0](S_i^r;p_ir)$ be the number of such convex $k$-gons in the other wedge with odd $k$.
	Then the number of even convex $k$-gons in that wedge is $R-1$.

	For polygons of type (iii), note that combining two polygons with the same parity in the number of vertices, we obtain a polygon with an odd number of vertices, while combining two polygons with different parities we obtain a polygon with an even number of vertices.
	Hence, the alternating sum for polygons of type (iii) is $(LR + (L-1)(R-1)) - (L(R-1) + (L-1)R) = (2LR -L-R+1)-(2LR-L-R)=1$.
	
	For polygons of type (ii), each polygon with an even number of vertices gives a polygon with an odd number of vertices when combined with the triangle $\Delta{p_iqr}$, and vice versa. 
	Hence, the alternating sum for polygons of type (ii) is $((L-1)+(R-1)) - (R+L) = -2$.

	Finally, for type (i) we only have the triangle $\Delta{p_iqr}$, which contributes $+1$ to the alternating sum.
	Hence altogether we obtain an alternating sum of $+1 -2 +1 = 0$ also for the third case, which completes the proof of the claim.
\end{proof}

Note that $\ASum[1](S;e)$, the alternating sum of convex polygons having $e$ as an edge and exactly one interior point, highly depends on the position of the points of $S$ and cannot be expressed by the number of extreme and non-extreme points of $S$.
On the other hand, $\ASum[1](S;e) = \sum_{p\in S} \ASum[1]^p(S;e)$ and hence we can use Lemma~\ref{one_fixed_edge_one_fixed_point}
to derive bounds for $\ASum[1](S;e)$, analogous to Lemma~\ref{edgelemma} for~$\ASum[0](S;e)$.

\begin{theorem}\label{one_fixed_edge_one_arb_point}
	For any set $S$ of $n$ points in general position in the plane, $h$ of them on the boundary of the convex hull, and any edge $e=qr$ of $S$ it holds that
	$\max\{h,4\}-n \leq \ASum[1](S;e) \leq 1$.
\end{theorem}

\begin{proof}
	For the upper bound, consider an arbitrary point $p\in S\setminus \{q,r\}$. 
	Reconsider the proof of Lemma~\ref{one_fixed_edge_one_fixed_point} and the $q$-wedge and the $r$-wedge. 
	Then, for $\ASum[1]^p(S;e)=1$ it is necessary that none of those wedges contains  points of $S$ in the interior. 
	As this can happen for at most one point of $S$ and as otherwise $\ASum[1]^p(S;e)\leq 0$, the upper bound follows.
	
	For the lower bound, note first that for any extreme point $p$ of $S$ we have $\ASum[1]^p(S;e)=0$. 
	Hence, $\sum_{p\in S} \ASum[1]^p(S;e) \geq (-1)\cdot(n-h)$, which for $h\geq 4$ is the claimed lower bound.
	Further, if $e$ is not a convex hull edge of $S$, then the points on the non-considered side of $e$ (there is at least one) can be ignored,
	again implying the claimed bound.
	So assume that $h=3$ and $e$ is a convex hull edge. Let $p$ be a non-extreme point of $S$ that has maximum distance to~$e$. 
	Then, the only convex polygon spanned by $S$ and having $e$ as an edge and that contains $p$ in its interior, is a triangle, since $S$ has $h=3$ extreme points. 
	This implies that $\ASum[1]^p(S;e)\in\{0,1\}$ and hence 
	$\sum_{p\in S} \ASum[1]^p(S;e) \geq (-1)\cdot(n-4)=4-n = \max\{4,h\}-n$.
\end{proof}
	
We remark that both bounds are tight in the sense that there exist arbitrary large point sets and edges obtaining them. 
For the lower bound, a quadrilateral with a concave chain of $n-4$ edges added close enough to one edge, and $e$ being the edge opposite to the chain gives an example; see Figure~\ref{fig:tight_bounds}. 
Actually, the latter also provides edges $f$ with $\ASum[1](S;f) = 1$; see Figure~\ref{fig:tight_bounds}. 
A different example for the upper bound is when the considered edge $g$ has a point $p$ sufficiently close to it. This 
guarantees that any non-empty polygon with $g$ as edge contains $p$, and hence $\ASum[1](S;g) = \ASum[0](S\setminus\{p\};g) = 1$. 

\begin{figure}[htb]
\centering
\includegraphics[page=2,scale=1]{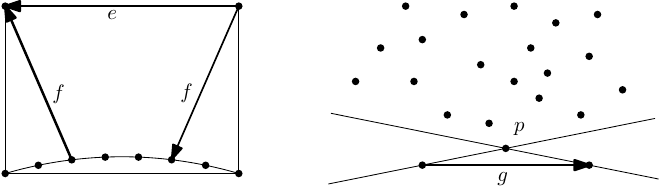} %
\caption{Examples reaching the lower and upper bounds of Theorem~\ref{one_fixed_edge_one_arb_point}: 
	$\ASum[1](S;e) = 4-n$ and $\ASum[1](S;f) = 1$ (left) and $\ASum[1](S;g) = 1$ (right).}
\label{fig:tight_bounds}
\end{figure}

\subsection{Inequalities}

Pinchasi et al.~\cite{prs_oecp_06} derived several inequalities that involve the parameters $X_{k,0}$. In this section, we show analogous inequalities for the parameters 
$X_{k,1}.$ We will need the following lemma proved in~\cite{prs_oecp_06}.

\begin{lemma}[\cite{prs_oecp_06}] \label{lm:m00_e_ineq}
For any set $S$ of $n$ points in general position and a directed edge $e$ spanned by two points of $S$, if there is at least one point of $S$ to the left of $e$, then it holds that 
\begin{itemize}
	\item for each $t \geq 3$ odd,
	$$
	X_{3,0}(S;e)-X_{4,0}(S;e)+X_{5,0}(S;e)-\ldots+X_{t,0}(S;e) \geq 1,
	$$
	\item for each $t \geq 4$ even,
	$$
	X_{3,0}(S;e)-X_{4,0}(S;e)+X_{5,0}(S;e)-\ldots-X_{t,0}(S;e) \leq 1,
	$$
	with equality holding, in either case, if and only if $X_{t+1,0}(S;e)=0$.
\end{itemize}
\end{lemma}

\begin{lemma}\label{lm:m01_p_ineq}
For any set $S$ of $n$ points in general position and a point $p \in S$ in the interior of the convex hull of $S$, it holds that 
\begin{itemize}
	\item for each $t \geq 3$ odd,
	$$
	X_{3,1}^p(S)-X_{4,1}^p(S)+\ldots+X_{t,1}^p(S) \geq 1,
		$$
	\item for each $t \geq 4$ even,
	$$
	X_{3,1}^p(S)-X_{4,1}^p(S)+\ldots-X_{t,1}^p(S) \leq 1,
		$$
	with equality holding, in either case, if and only if $X_{t+1,1}^p(S)=0$.
\end{itemize}
\end{lemma}
\begin{proof}
First of all, by Lemma~\ref{lem:one_fixed_point}, we have that if $X_{t+1,1}^p(S)=0$, then equality holds.
Let us then consider that $p$ is close enough to a convex hull edge $e$ of~$S$. We can apply the same argument as in the proof of Lemma~\ref{lem:one_fixed_point}.  We get
\begin{equation*}
\left.\begin{aligned}
 &X_{3,1}^p(S)-X_{4,1}^p(S)+\ldots+(-1)^{t+1}X_{t,1}^p(S)\\  &= X_{3,0}(S \setminus \left\{p\right\}; e)-X_{4,0}(S \setminus \left\{p\right\}; e)+\ldots+(-1)^{t+1}X_{t,0}(S \setminus \left\{p\right\}; e),
\end{aligned}\right.
\end{equation*}
and the result of Lemma~\ref{lm:m00_e_ineq} applies.
Moreover, if we consider a continuous motion of $p$ which is sufficiently generic, proceeding exactly in the same manner as in the proof of Lemma~\ref{lem:one_fixed_point}, the value of $X_{3,1}^p(S)-X_{4,1}^p(S)+\ldots+(-1)^{t+1}X_{t,1}^p(S)$ does not change if $p$ is located arbitrarily inside $S$. 
\end{proof}

\begin{theorem}\label{thm:m01_ineq}
For any set $S$ of $n$ points in general position in the plane, $h$ of them on the boundary of the convex hull, we have
\begin{itemize}
	\item for each $t \geq 3$ odd,
	$$
	X_{3,1}(S)-X_{4,1}(S)+X_{5,1}(S)-\ldots+X_{t,1}(S) \geq n-h,
	$$
	\item for each $t \geq 4$ even,
	$$
	X_{3,1}(S)-X_{4,1}(S)+X_{5,1}(S)-\ldots-X_{t,1}(S) \leq n-h,
	$$
	with equality holding, in either case, if and only if $X_{t+1,1}(S)=0$.
\end{itemize}
\end{theorem}
\begin{proof}
Observe that, by Theorem~\ref{thm:one_point}, if $X_{t+1,1}=0$ then the equality holds.
Using the same argument as in Theorem \ref{thm:one_point}, it is clear that
\begin{equation*}
\left.\begin{aligned}
 &X_{3,1}(S)-X_{4,1}(S)+X_{5,1}(S)-\ldots+(-1)^{t+1}X_{t,1}(S)\\  
& = \sum_{p \in S} \left(X_{3,1}^p(S)-X_{4,1}^p(S)+X_{5,1}^p(S)-\ldots+(-1)^{t+1}X_{t,1}^p(S)\right).
\end{aligned}\right.
\end{equation*}
Let $t$ be odd. By Lemma \ref{lm:m01_p_ineq}, $X_{3,1}^p(S)-X_{4,1}^p(S)+\ldots+X_{t,1}^p(S) \geq 1$ if $p$ is an interior point of $S$. Therefore,
\[
X_{3,1}(S)-X_{4,1}(S)+X_{5,1}(S)-\ldots+X_{t,1}(S)  =  \sum_{p \in S} \left(X_{3,1}^p(S)-X_{4,1}^p(S)+\ldots+X_{t,1}^p(S)\right) \geq n - h.
\]
If equality holds, then $X_{3,1}^p(S)-X_{4,1}^p(S)+X_{5,1}^p(S)-\ldots+X_{t,1}^p(S)=1$ if $p$ is an interior point of~$S$ and, by Lemma \ref{lm:m01_p_ineq}, $X_{t+1,1}^p(S)=0$ for these points. Therefore, $X_{t+1,1}^p(S)=0$ for all the points and, in consequence, $X_{t+1,1}(S)=0$.

If $t$ is even, the proof is analogous with the unique difference that the inequality is $X_{3,1}^p(S)-X_{4,1}^p(S)+X_{5,1}^p(S)-\ldots+(-1)^{t+1}X_{t,1}^p(S) \leq 1$ if $p$ is an interior point of $S$. Therefore, the proof proceeds in the same manner but the direction of the inequalities is reversed.
\end{proof}

From Theorem \ref{thm:m01_ineq} with $t=4$ we obtain the following corollary.
\begin{corollary}
$X_{4,1} \geq X_{3,1} -n + h$.
\end{corollary}

\section{Weighted sums}\label{sec:weighted_sums} %

In this section, we consider sums of the form $\FSum(S) = \sum_{k\geq 3} \sum_{\ell\geq 0} \fkl{k}{\ell} \Xkl(S)$. 
The following theorem is the main tool used throughout this section and the next one.

\begin{theorem}\label{thm:fkl}
	For any function $\fkl{k}{\ell}$ that fulfills the equation
	\begin{equation} \label{equ:fkl}
		\fkl{k}{\ell} = \fkl{k+1}{\ell-1} + \fkl{k}{\ell-1},
	\end{equation}
	the sum $\FSum(S) := \sum_{k\geq 3} \sum_{\ell\geq 0} \fkl{k}{\ell} \Xkl(S)$
	is invariant over all sets $S$ of $n$ points in general position, %
	that is, $\FSum(S)$ only depends on the cardinality of $S$.
\end{theorem}
\begin{proof}
Consider a point set $S$ in general position. We claim that any continuous motion of the points of $S$ which is sufficiently generic does not change the value of $F(S)$. 

Consider that $p,q,r \in S$ become collinear, with $r$ lying between $p$ and $q$; thus the only convex polygons spanned by $S$ that may change are those that have $pq$ as an edge and $r$ in its interior or those that have $p,q,$ and $r$ as vertices. 

Let $Q$ be a convex $k$-gon with $\ell$ interior points that contains $pq$ as an edge and $r$ in its interior. If $r$ moves outside of $Q$, then $Q$ has $\ell-1$ points in its interior and the $(k+1)$-gon~$Q'$ obtained by replacing the edge $pq$ of $Q$ by the polygonal path $prq$, starts being convex with $\ell-1$ points in its interior. Hence, in this movement we can assign each polygon counted in $X_{k,\ell}$ (which disappears) to one polygon counted in $X_{k,\ell-1}$ and to one counted in $X_{k+1,\ell-1}$ (which appear). Since $f(k,\ell)=f(k+1,\ell-1)+f(k,\ell-1)$, the movement of the point does not change the value of $F(S)$.

Symmetrically, if $r$ moves inside $Q$ (with $\ell$ points in its interior), then $Q$ has $\ell+1$ points in its interior and the $(k+1)-$gon $Q'$, with also $\ell$ points in its interior, stops being convex. Again, this does not change the value of $F(S)$.
\end{proof}

\begin{observation}\label{obs:lincomb}
	Let $f_1(k, \ell)$ and $f_2(k, \ell)$ be two functions that fulfill Equation~(\ref{equ:fkl}). Then every linear combination of $f_1(k,\ell)$ and $f_2(k,\ell)$ fulfills Equation~(\ref{equ:fkl}) as well.
\end{observation}

Using Theorem~\ref{thm:fkl}, we first derive several relations for the sum over all convex polygons, weighted by their number of interior points. Each of this relations is proved by showing that the function $f(k,\ell)$ satisfies~Equation~(\ref{equ:fkl}) and by evaluating $F(S)$ for a set of $n$ points in convex position. Note that for point sets in convex position $X_{k,\ell}=0$ for $\ell\geq 1$.

\begin{corollary}\label{cor:is}
	For any set $S$ of $n\geq 3$ points in general position, it holds that  
	\[  \sum_{k=3}^{n} \sum_{\ell=0}^{n-3} 2^\ell \Xkl(S) = 2^n - \frac{n^2}{2} - \frac{n}{2} - 1. \]
\end{corollary}
\begin{proof}
	Let $\fkl{k}{\ell} = 2^\ell$. 
	Then $2^\ell = 2^{\ell-1} +  2^{\ell-1} = \fkl{k+1}{\ell-1} + \fkl{k}{\ell-1}$.
	For a set $S$ of $n$ points in convex position we get
	$\sum_{k=3}^n 2^0 \Xkl[k,0] = \sum_{k=3}^n \binom{n}{k} = 2^n - \binom{n}{2} - \binom{n}{1} - \binom{n}{0} = 2^n - \frac{n^2}{2} - \frac{n}{2} - 1$.
\end{proof}

\begin{corollary}\label{cor:bms}
	For any set $S$ of $n$ points in general position, and every integer $3 \leq m \leq n$ it holds that  
	\[ \sum_{k=3}^{m} \sum_{\ell=m-k}^{n-k} \binom{\ell}{m-k} \Xkl(S) = \binom{n}{m}. \]
\end{corollary}
\begin{proof}
	Let $\fkl{k}{\ell} =  \binom{\ell}{m-k}$ and note that $m$ is fixed. 
	Then 
	\[ \fkl{k}{\ell} = \binom{\ell}{m-k} = \binom{\ell-1}{m-k-1} + \binom{\ell-1}{m-k} = \fkl{k+1}{\ell-1} + \fkl{k}{\ell-1}.\] 
	For a set $S$ of $n$ points in convex position we get
		$ \sum_{k=3}^{m} \binom{0}{m-k} \Xkl[k,0](S) = \Xkl[m,0](S) = \binom{n}{m}$.
\end{proof}

\noindent
Let $\{\Fib(n)\}_{n \in \mathbb{Z}}$ be the sequence of Fibonacci numbers, satisfying the recurrence relation \mbox{$\Fib(n)=\Fib(n-1)+\Fib(n-2)$}, with $\Fib(0)=0$ and $\Fib(1)=1$. 
\begin{corollary}\label{cor:fib1}
	For any set $S$ of $n\geq 3$ points in general position it holds that
	\[\sum_{k=3}^{n} \sum_{\ell=0}^{n-3} \Fib(k+2\ell) X_{k,\ell}(S) =\Fib(2n)-n-\binom{n}{2}.\]
\end{corollary}
\begin{proof}
The function $f(k,\ell)=\Fib(k+2\ell)$ satisfies~Equation~(\ref{equ:fkl}) by definition of the Fibonacci numbers.
Consider then a set $S$ of $n$ points in convex position and use the following identity; see~\cite{spivey}. 
	\[\sum_{k=0}^{n}\Fib(k)\binom{n}{k}=\Fib(2n).\]
Then,
	\[\sum_{k=3}^{n}  \Fib(k+0) X_{k,0}(S) = \sum_{k=3}^{n} \Fib(k) \binom{n}{k}  = \Fib(2n)-n-\binom{n}{2}.\]
\end{proof}

\begin{corollary}\label{cor:fib2}
	For any set $S$ of $n\geq 3$ points in general position it holds that
	\[\sum_{k=3}^{n} \sum_{\ell=0}^{n-3} (-1)^{k+\ell}\Fib(k-\ell) X_{k,\ell}(S) =-\Fib(n)+n-\binom{n}{2}.\]
\end{corollary}
\begin{proof}
The function $f(k,\ell)=(-1)^{k+\ell}\Fib(k-\ell)$ satisfies~Equation~(\ref{equ:fkl}) by definition of the Fibonacci numbers.
Consider then a set $S$ of $n$ points in convex position and use the following identity; see~\cite{spivey}.
	\[\sum_{k=0}^{n}(-1)^k \Fib(k)\binom{n}{k}=-\Fib(n).\]
Then, 
	\[\sum_{k=3}^{n} (-1)^{k+0} \Fib(k-0) X_{k,0}(S) = \sum_{k=3}^{n} (-1)^{k} \Fib(k) \binom{n}{k} = -\Fib(n)+n-\binom{n}{2}.\]
\end{proof}

\begin{corollary}\label{cor:cheby}
	Any set $S$ of $n\geq 3$ points in general position satisfies the following equations.
\begin{equation}\label{equ:ChebyT}
\sum_{k=3}^{n} \sum_{\ell=0}^{n-3} 2\cos\left(\frac{(2k+\ell)\pi}{3}\right) X_{k,\ell}(S) = \binom{n}{2}+n-2+ 2\cos\left(\frac{n \pi}{3}\right) 
\end{equation}

\begin{equation}\label{equ:ChebyU}
\sum_{k=3}^{n} \sum_{\ell=0}^{n-3} \frac{2}{\sqrt{3}}\sin\left(\frac{(2k+\ell)\pi}{3} \right) X_{k,\ell}(S) = \binom{n}{2}-n+ \frac{2}{\sqrt{3}}\sin\left(\frac{n \pi}{3}\right)
\end{equation}
\end{corollary}

\begin{proof}
	Equations~(\ref{equ:ChebyT}) and~(\ref{equ:ChebyU}) are obtained using Chebyshev polynomials, see e.g.~\cite{mason}.
The Chebyshev polynomials of the first kind, $T_m(x):=\cos(m \theta)$ with $x=\cos(\theta)$, 
	satisfy the recurrence relation $T_m(x)=2x T_{m-1}(x) - T_{m-2}(x).$
For $x=\frac{1}{2}$, this gives the relation
	\[\cos\left(  \frac{(m-1)\pi}{3}\right)= \cos\left(\frac{m \pi}{3}\right) + \cos\left( \frac{(m-2) \pi}{3}\right).\]
Now set $m=2k+\ell+1$ and observe that the function $f_1(k,\ell):= 2\cos\left(\frac{(2k+\ell)\pi}{3}\right)$ 
	satisfies~Equation~(\ref{equ:fkl}).
Consider then a set $S$ of $n$ points in convex position. 
We use a binomial identity, which can be found in~\cite{binomial}, Equation (1.26),
	\[ \sum_{k=0}^{n} \binom{n}{k}\cos(k y) = 2^n   \cos\left(\frac{n y}{2}\right)\left(\cos\left(\frac{y}{2}\right)\right)^n.\]
With $y=\frac{2\pi}{3}$ and hence $\cos(\frac{y}{2})= \frac{1}{2}$ we obtain Equation~(\ref{equ:ChebyT}):
\begin{eqnarray*}
\sum_{k=3}^{n}  f_1(k,0) X_{k,0}(S)  
	& = & \sum_{k=0}^{n} 2\cos\left(\frac{2k \pi}{3}\right)\binom{n}{k} - \sum_{k=0}^{2} 2\cos\left(\frac{2k\pi}{3}\right) \binom{n}{k}\\
	& = & 2\cos\left(\frac{n \pi}{3}\right) + \binom{n}{2}+n-2. 
\end{eqnarray*}
To prove Equation~(\ref{equ:ChebyU}), consider the Chebyshev polynomial of the second kind,
	\[U_m(x):=\frac{\sin((m+1)\theta)}{\sin(\theta)},\]
with $x=\cos(\theta)$. $U_m(x)$ satisfies  $U_m(x)=2x U_{m-1}(x) - U_{m-2}(x)$.
For $x=\frac{1}{2}$, this gives the relation
	\[\sin\left(  \frac{m \pi}{3}\right)=  \sin\left( \frac{(m+1) \pi}{3}\right) + \sin\left(\frac{(m-1) \pi}{3}\right).\]
Now set $m=2k+\ell$ and observe that the function $f_2(k,\ell):= \frac{2}{\sqrt{3}}\sin\left(\frac{(2k+\ell)\pi}{3}\right)$ 
	satisfies~Equation~(\ref{equ:fkl}).
Consider then a set $S$ of $n$ points in convex position.
We use a binomial identity, which can be found in~\cite{binomial}, Equation (1.27).
	\[ \sum_{k=0}^{n} \binom{n}{k}\sin(k y) = 2^n   \sin\left(\frac{n y}{2}\right)\left(\cos\left(\frac{y}{2}\right)\right)^n.\]
This identity with $y=\frac{2\pi}{3}$ completes the proof of Equation~(\ref{equ:ChebyU}):
\begin{eqnarray*} 
	\sum_{k=3}^{n}  f_2(k,0) X_{k,0}(S) 
	& = & \sum_{k=0}^{n} \frac{2}{\sqrt{3}} \sin\left(\frac{2k \pi}{3}\right) \binom{n}{k} -  \sum_{k=0}^{2} \frac{2}{\sqrt{3}} \sin\left(\frac{2k \pi}{3}\right) \binom{n}{k} \\
	& = & \frac{2}{\sqrt{3}}\sin\left(\frac{n \pi}{3}\right) + \binom{n}{2}-n. 
\end{eqnarray*}\end{proof}
Next we show that the functions $f(k,\ell)$ can be expressed as a linear combination of the functions $f(k+i,0)$, for $i=0,\ldots,\ell.$
\begin{theorem}\label{thm:general}
The general solution of the recurrence relation $\fkl{k}{\ell} = \fkl{k+1}{\ell-1} + \fkl{k}{\ell-1}$ is given by the equation
\begin{equation}
f(k,\ell) = \sum_{i=0}^{\ell} \binom{\ell}{i} \ f(k+i,0).
\label{eq:rr_general}
\end{equation}
\end{theorem}

\begin{proof}
We prove the statement by induction on $\ell$. It is straightforward to verify the base cases $\ell=0$ and $\ell=1$. For the inductive step, let $\lambda \geq 0$ be given and suppose (\ref{eq:rr_general}) holds for $\ell = \lambda$. Then,
\begin{eqnarray*}
f(k,\lambda+1) &=& f(k+1,\lambda)+f(k,\lambda) = \sum_{i=0}^{\lambda} \binom{\lambda}{i} f(k+1+i,0)+\sum_{i=0}^{\lambda} \binom{\lambda}{i} f(k+i,0) = \nonumber\\
&=& \sum_{i=1}^{\lambda+1} \binom{\lambda}{i-1} f(k+i,0)+\sum_{i=0}^{\lambda} \binom{\lambda}{i} f(k+i,0) = \nonumber \\
& =& f(k,0)+\sum_{i=1}^{\lambda}\left[\binom{\lambda}{i}+ \binom{\lambda}{i-1} \right] f(k+i,0)+f(k+i,0) = \nonumber\\
&=& \sum_{i=0}^{\lambda+1} \binom{\lambda+1}{i} \ f(k+i,0).
\end{eqnarray*}
\end{proof}

Note that there may be many functions $f(k,\ell)$ that fulfill Equation (\ref{equ:fkl}), thus many different sums $F(S)$ that only depend on $n$. An interesting point of view is to analyze how many of them are independent.
\begin{proposition}\label{prop:indy}
The maximum number of linearly independent equations of the form 
$F_j(S) =\sum_{k \geq 3} \sum_{\ell \geq 0} f_j(k,\ell) X_{k,\ell}$, where each $f_j(k,\ell)$ satisfies Equation (\ref{eq:rr_general}), in terms of the variables $X_{k,\ell}$ is $n-2$.
\end{proposition}
\begin{proof}
Let $F_j(S) = \sum_{k=3}^{n} \sum_{\ell = 0}^{n-k} f_j(k,\ell) X_{k,\ell}$, for $1 \leq j \leq n-1$, be sums that only depend on~$n$, where $f_j(k,\ell)$ satisfies (\ref{eq:rr_general}) for all $j$. Let us consider the matrix

\begin{equation*}M=
\left( \begin{array}{ccccccc}
 f_1(3,0) & f_1(3,1) & \cdots & f_1(k,\ell) & \cdots & f_1(n-1,1) & f_1(n,0)\\
 f_2(3,0) & f_2(3,1) & \cdots & f_2(k,\ell) & \cdots & f_2(n-1,1) & f_2(n,0)\\
\vdots & \vdots & \vdots & \vdots & \vdots & \vdots & \vdots \\
 f_j(3,0) & f_j(3,1) & \cdots & f_j(k,\ell) & \cdots & f_j(n-1,1) & f_j(n,0)\\
\vdots & \vdots & \vdots & \vdots & \vdots & \vdots & \vdots \\
 f_{n-1}(3,0) & f_{n-1}(3,1) & \cdots & f_{n-1}(k,\ell) & \cdots & f_{n-1}(n-1,1) & f_{n-1}(n,0)
\end{array} \right),
\label{eq:matrix1}
\end{equation*}
whose entries are the $f_j(k,\ell)$ for $3 \leq k \leq n$, $0 \leq \ell \leq n-k,$ and $1 \leq j \leq n-1.$
The function $f_j$ occupies the entries of row number $j$ of $M$. Let $C[h]$ be column number $h$ of $M$. The first column $C[1]$ contains the functions $f_j(k,\ell)$ for values $k=3$ and $\ell=0$, then the following columns are for the values $k=3$ and $\ell \geq 1$ in ascending order, then there are columns containing the functions with values $k=4$ and $\ell \geq 0$ in ascending order, and so on.  This order implies that, for given values $k$ and $\ell$, matrix entries $f_j(k,\ell)$ are in column $C[h]$ of $M$, where $h=\sum_{i=3}^{k-1}(n-i+1)+\ell+1$.\\  
It is sufficient to show that the matrix $M$ has rank $n-2$. Since each $f_j(k,\ell)$ satisfies (\ref{eq:rr_general}), we can express matrix $M$ as follows.

\begin{equation*}
M=
\left( \begin{array}{ccccc}
 f_1(3,0) & \cdots & \sum_{i=0}^{\ell} \binom{\ell}{i} \ f_1(k+i,0) & \cdots & f_1(n,0)\\
 f_2(3,0) & \cdots & \sum_{i=0}^{\ell} \binom{\ell}{i} \ f_2(k+i,0) & \cdots & f_2(n,0)\\
\vdots & \vdots & \vdots & \vdots & \vdots \\
 f_j(3,0) & \cdots & \sum_{i=0}^{\ell} \binom{\ell}{i} \ f_j(k+i,0) & \cdots & f_j(n,0)\\
\vdots & \vdots & \vdots & \vdots & \vdots \\
 f_{n-1}(3,0) & \cdots & \sum_{i=0}^{\ell} \binom{\ell}{i} \ f_{n-1}(k+i,0) & \cdots & f_{n-1}(n,0)
\end{array} \right).
\label{eq:mat1}
\end{equation*}

We observe that the columns corresponding to the functions with $\ell>0$ depend on the $n-2$ columns with $\ell=0$. With adequate operations these columns can be transformed to columns of zeros. That is, if we change each column $C[h]$ corresponding to the functions with $\ell>0$ in the following manner: 
\begin{equation*}
C\left[{\sum_{s=3}^{k-1}(n-s+1)+\ell+1}\right] \ \rightarrow \  
C\left[{\sum_{s=3}^{k-1}(n-s+1)+\ell+1}\right] - \sum_{i=0}^{\ell} \binom{\ell}{i} C\left[{\sum_{s=3}^{k+i-1}(n-s+1)+1}\right],
\end{equation*}
we obtain the matrix

\begin{equation*}
\left( \begin{array}{ccccccc}
 f_1(3,0) & 0 & \cdots & f_1(k,0) & \cdots & 0 & f_1(n,0)\\
 f_2(3,0) & 0 & \cdots & f_2(k,0) & \cdots & 0 & f_2(n,0)\\
\vdots & \vdots & \vdots & \vdots & \vdots & \vdots & \vdots \\
 f_j(3,0) & 0 & \cdots & f_j(k,0) & \cdots & 0 & f_j(n,0)\\
\vdots & \vdots & \vdots & \vdots & \vdots & \vdots & \vdots \\
 f_{n-1}(3,0) & 0 & \cdots & f_{n-1}(k,0) & \cdots & 0 & f_{n-1}(n,0)
\end{array} \right).
\label{eq:mat}
\end{equation*}

Thus, there are only $n-2$ non-zero columns corresponding to the functions $f_j(k,0)$, implying that the maximum possible rank of $M$ is $n-2$.
\end{proof}

Using the general solution of Theorem \ref{thm:general}, we derive further relations for sums over all convex polygons.
\begin{corollary}\label{cor:ps}
For any point set $S$ of $n \geq 3$ points in general position and for any $x \in \R$, it holds that
\begin{equation}\label{eq:ps}
P_x(S) := \sum_{k=3}^n \sum_{\ell=0}^{n-k} x^k \left(1+x\right)^\ell X_{k,\ell} = \left(1+x\right)^n -1 - x\cdot n - x^2 \binom{n}{2}.
\end{equation}
\end{corollary}

\begin{proof}
Define $f(k+i,0)=x^{k+i}$, then
$$
f(k,\ell)=\sum_{i=0}^{\ell} \binom{\ell}{i} f(k+i,0) = x^k \sum_{i=0}^{\ell} \binom{\ell}{i} x^i = x^k \cdot \left(1+x\right)^\ell.
$$
The result then follows by considering a set of $n$ points in convex position and we have
$$
\sum_{k=3}^n x^k \binom{n}{k} = (1+x)^n -1 - x \cdot n - x^2 \cdot \binom{n}{2}.
$$\end{proof}

\begin{observation}
Corollary \ref{cor:is} is the particular case $x=1$ of Corollary \ref{cor:ps}.
\end{observation}

As shown in Proposition \ref{prop:indy}, the maximum number of possible linearly independent equations is $n-2$. From Corollary \ref{cor:ps}, we can obtain multiple equations of this form. Therefore, we evaluate the independence of these equations.

\begin{proposition}\label{prop:indy2}
	For any point set $S$ of $n \geq 3$ points in general position, let $x_j \in \R_{\neq 0}$, with $1 \leq j \leq n-2$, be distinct values, and consider the $n-2$ equations $P_j(S) =  \sum_{k=3}^n \sum_{\ell=0}^{n-k} x_j^k \left(1+x_j\right)^{\ell} X_{k,\ell}$.  Then these equations are linearly independent.
\end{proposition}

\begin{proof}
Using the argument of Proposition \ref{prop:indy}, it is sufficient to analyze the columns corresponding to the variables $X_{k,0}$. Thus, we consider the matrix

\begin{equation*}
\left( \begin{array}{ccccccc}
 x_1^3 & x_1^4 & \cdots & x_1^k & \cdots & x_1^n\\
 x_2^3 & x_2^4 & \cdots & x_2^k & \cdots & x_2^n\\
\vdots & \vdots & \vdots & \vdots & \vdots \\
 x_j^3 & x_j^4 & \cdots & x_j^k & \cdots & x_j^n\\
\vdots & \vdots & \vdots & \vdots & \vdots \\
 x_{n-2}^3 & x_{n-2}^4 & \cdots & x_{n-2}^k & \cdots & x_{n-2}^n\\
\end{array} \right).
\label{eq:vand3}
\end{equation*}

Now, we can divide each row $j$ by $x_j^3$ and we obtain

\begin{equation*}
\left( \begin{array}{ccccccc}
 1 & x_1 & \cdots & x_1^{k-3} & \cdots & x_1^{n-3}\\
 1 & x_2 & \cdots & x_2^{k-3} & \cdots & x_2^{n-3}\\
\vdots & \vdots & \vdots & \vdots & \vdots \\
 1 & x_j & \cdots & x_j^{k-3} & \cdots & x_j^{n-3}\\
\vdots & \vdots & \vdots & \vdots & \vdots \\
 1 & x_{n-2} & \cdots & x_{n-2}^{k-3} & \cdots & x_{n-2}^{n-3}\\
\end{array} \right).
\label{eq:vand}
\end{equation*}

This matrix is an $(n-2) \times (n-2)$ Vandermonde matrix and all $x_i$ are distinct. Therefore the rank of the matrix is $n-2$, implying that the $n-2$ equations are linearly independent.
\end{proof}

The following result is an immediate consequence of Proposition~\ref{prop:indy}, Corollary~\ref{cor:ps}, and Proposition~\ref{prop:indy2}.

\begin{corollary}\label{cor:last}
Let $S$ be a set of $n$ points in general position.
Then any sum of the form $\sum_{k \geq 3} \sum_{\ell \geq 0} f(k,\ell) X_{k,\ell}$, where the function $f(k,\ell)$ fulfills Equation (\ref{equ:fkl}), can be expressed by $n-2$ sums of the form (\ref{eq:ps}) with distinct values $x \in \R$.
\end{corollary}

\section{Moment sums}\label{sec:moment_sums} %

Recall that Theorem~\ref{thm:one_point} states that for any set $S$ of $n$ points in general position, the alternating sum of the numbers of polygons with one interior point is $\ASum[1](S) = n-h$. 
Combining this with the first moment of the numbers of empty polygons $M_1(S) = 2\binom{n}{2}-h$ from~\cite{prs_oecp_06}, we observe that the difference $M_1(S) - \ASum[1](S) = 2\binom{n}{2} - n$ is again a function that only depends on the cardinality $n$ of~$S$ and hence is independent of the combinatorics of the underlying point set~$S$. 
In this section, we show that the latter observation can actually be extended to moment sums for convex polygons with at most two interior points. 

\begin{theorem}\label{thm:moment_sums}
	For any set $S$ of $n$ points in general position and for integers $0 \leq r\leq 2$, it holds that
	\[
		\FrSum(S) := \sum_{ k\geq 3} \sum_{\ell=0}^r (-1)^{k-\ell+1} \mrk[r-\ell]{k-\ell} \Xkl(S) = 
			\left\{\begin{array}{ll}
					\binom{n}{2}-n+1 		& \mbox{for}\ r=0 \\
					2\binom{n}{2}-n 		& \mbox{for}\ r=1 \\
					-\binom{n}{2}+n				& \mbox{for}\ r= 2 \\ 
			\end{array}\right.
	\]
\end{theorem}
\begin{proof}
	For $r \in \{0,1\}$, the statement follows from the results of~\cite{prs_oecp_06} and Theorem~\ref{thm:one_point}.
	For the case $r=2$, we have 
	$$F_2(S)= \sum_{k \geq 3} \sum_{\ell=0}^{2} (-1)^{k-\ell+1} m_{2-\ell}(k-\ell)\Xkl(S)$$
	$$= \sum_{k \geq 3} (-1)^{k+1}\frac{k(k-3)}{2}X_{k,0}+(-1)^k(k-1)X_{k,1}+(-1)^{k+1}X_{k,2}. $$
	Let %
	\[ f(k,\ell)  :=  \left\{\begin{array}{ll}
						(-1)^{k-\ell+1} m_{2-\ell}(k-\ell) & \mbox{for}\ 0 \leq \ell \leq 2 \\
						0 & \mbox{for}\ 2 < \ell \\
						\end{array}\right.
	\]
	Using $f(k,\ell)$, we can express the sum as $F_2(S) = \sum_{k\geq 3} \sum_{\ell\geq 0} f(k,\ell) \Xkl(S)$.

	We first show that $F_2(S)$ only depends on the cardinality $n$ of $S$ by proving that the function $f(k,\ell)$ fulfills Equation~(\ref{equ:fkl}) from Theorem~\ref{thm:fkl}, namely, for all $\ell>0$,  $f(k+1,\ell-1) + f(k,\ell-1) = f(k,\ell)$.
	For $\ell > 3$, Equation~(\ref{equ:fkl}) is trivially true as all terms in the equation are equal to zero. 
	For $\ell=3$, we have $f(k,\ell) = 0$ and
	\begin{eqnarray*}
		f(k+1,\ell-1) + f(k,\ell-1)	& = & (-1)^{(k+1) - (\ell-1) + 1} m_{2-(\ell-1)}((k+1)-(\ell-1)) \\
					&& + (-1)^{k-(\ell-1)+1} m_{2-(\ell-1)}(k-(\ell-1)) \\
		& = & (-1)^{k-\ell+3} m_{0}(k-\ell+2) + (-1)^{k-\ell+2} m_{0}(k-\ell+1) \ = \  0. \\
	\end{eqnarray*}
	For $0 < \ell \leq 2$, Equation~(\ref{equ:fkl}) for $f(k,\ell)$ can be stated in terms of the function $\mrk{k}$, which gives 
	\begin{eqnarray}
		\label{eqn:mrkl}
		m_{2-\ell+1}(k-\ell+2) - m_{2-\ell+1}(k-\ell+1) & = & m_{2-\ell}(k-\ell).
	\end{eqnarray}
	We consider the case $\ell=2$, $\ell=1$ and $\ell=0$ separatedly.
	\begin{itemize}
	\item Let $\ell=2$. The truth of Equation~(\ref{eqn:mrkl}) follows directly from $\mrk[1]{k}=k$ for $k \geq 2$ and $\mrk[0]{k}=1$.
	\item Let $\ell=1$. If $k=3$, then $m_2(k)=0$ and $m_2(k+1)=m_1(k-1)=2$; the statement holds. If $k\geq 4$, then $m_{2}(k+1)-m_{2}(k)=\frac{(k+1)}{2}(k-2)-\frac{k}{2}(k-3)=k-1=m_1(k-1)$; the statement holds. 
	\item Let $\ell=0$. If $k=3$ then $m_3(k+2)=m_3(k+1)=m_2(k)=0$; the statement holds. If $k=4$, then $m_3(k+1)=0$ and $m_3(k+2)=m_2(k-1)=2$; the statement holds. 
			If $k\geq 5$, then $m_{3}(k+2)-m_{3}(k+1)=\frac{k+2}{3}\binom{k-2}{2}-\frac{k+1}{3}\binom{k-3}{2}=\frac{k}{2}(k-3)=m_2(k)$; the statement holds.
			\end{itemize}
	This finishes the proof that $F_2(S)$ only depends on the cardinality $|S|=n$.
		So what remains to show is that $F_2(S) = -\binom{n}{2}+n$. 
	To this end, let $S$ be set of~$n$ points in convex position.
	Then, $X_{k,0}(S)=\binom{n}{k}$ and $X_{k,\ell}(S)=0$ for $\ell>0$. We get
	$$F_2(S)= \sum_{k \geq 3} \sum_{\ell=0}^{2} (-1)^{k-\ell+1}m_{2-\ell}(k-\ell)\Xkl(S)=
	 \sum_{k=3}^{n} (-1)^{k+1}\frac{k(k-3)}{2}\binom{n}{k}.$$
	Since $\sum_{k=0}^{n}(-1)^{k}k\binom{n}{k}=0$ and $\sum_{k=0}^{n}(-1)^{k}k^2\binom{n}{k}=0$, see~\cite{spivey}, we have
	$$F_2(S)=0-\sum_{k=0}^{2} (-1)^{k+1}\frac{k(k-3)}{2}\binom{n}{k}=-\binom{n}{2}+n.$$

\end{proof}

\section{Higher dimensions}\label{sec:higher_dimensions}

In this section, we consider the generalization of the previous results to $d$-dimensional Euclidean space $\R^d$, for $d\geq 3$.
To this end, let $S$ be a set of $n\geq d+1$ points in $\R^d$ in general position, that is, no hyperplane contains more than $d$ points of~$S$.
 We again denote by $h$ the number of extreme points of $S$, that is, the points of~$S$ that lie on the boundary of the convex hull of~$S$. 
For each $k\geq d+1$, let $\Xk(S)$ be the number of empty convex $k$-vertex polytopes spanned by~$S$, and let $\Xkl(S)$ be the number of convex $k$-vertex polytopes of $S$ that have exactly $\ell$ points of~$S$ in their interior.

In~\cite{prs_oecp_06}, the authors extend their results on alternating moments of point sets in the Euclidean plane to point sets $S$ in $\R^d$.
They show %
that %
\begin{eqnarray*}
	M_0(S) \ := & \mathlarger{\sum}\limits_{k\geq d+1} (-1)^{k+d+1} \Xk(S) & = \ \sum\limits_{k=0}^d (-1)^{d-k} \binom{n}{k} %
																						\hspace{1.6cm} \mbox{(\cite{prs_oecp_06}, Theorem 4.1.)}	\\
	M_1(S) \ := & \mathlarger{\sum}\limits_{k\geq d+1} (-1)^{k+d+1} k \Xk(S) & = \ \sum\limits_{k=0}^d (-1)^{d-k} k\binom{n}{k} \ + \ i %
																						\hspace{0.5cm} \mbox{(\cite{prs_oecp_06}, Theorem 4.2.)}	%
\end{eqnarray*}
where $M_r(S)$ is again called the \emph{$r$-th alternating moment} of $\{\Xk(S)\}_{k\geq d+1}$ and $i=n-h$ is the number of interior points of $S$. 
The proofs of~\cite{prs_oecp_06} make use of a continous motion argument of points in $\mathbb{R}^d$. 
The same argument is applied in the following, so we refer the reader also to~\cite{prs_oecp_06}.  

\subsection{Alternating sums}\label{sec:alt_sums_rd}

For an oriented facet $f$, let $\Xk(S;f)$ be the number of empty convex $k$-vertex polytopes that have $f$ as a boundary facet and lie on the positive side of $f$.
In the proof of Theorem 4.2 in~\cite{prs_oecp_06}, the authors generalize Lemma~\ref{lem:one_fixed_point} to $\R^d$. The following lemma is implicit in~\cite{prs_oecp_06}.

\begin{lemma}[\cite{prs_oecp_06}]\label{facetlemma}
	For any set $S$ of $n \geq d+1$ points in general position in $\R^d$ and any oriented facet $f$ spanned by $d$ points of $S$, 
	it holds that
	\[ 
		\ASum[0](S;f) := \sum_{k\geq d+1} (-1)^{k+d+1} \Xk(S;f) \ = \ 
		\left\{\begin{array}{ll}
					1 & \mbox{ if $f$ has at least one point of $S$ } \\
					  & \mbox{ on the positive side} \\
					0 & \mbox{ otherwise.} \\
					\end{array}\right.
	\]
\end{lemma}

In accordance to the planar case, we denote by 
$\ASum[1]^p(S)$ the alternating sum of convex polytopes that contain exactly $p$ in their interior.
Using Lemma~\ref{facetlemma}, Lemma~\ref{lem:one_fixed_point} can be generalized to the following.

\begin{lemma}\label{lem:one_fixed_point_rd}
	For any set $S$ of $n\geq d+1$ points in general position in $\R^d$ and any point $p \in S$ it holds that 
	\[ 
		\ASum[1]^p(S) = \left\{\begin{array}{ll}
					0 & \mbox{ if $p$ is an extreme point of $S$} \\
					1 & \mbox{ otherwise.} \\
					\end{array}\right.
	\]
\end{lemma}
\begin{proof}
	If $p$ is an extreme point of $S$, then it cannot be in the interior of any convex polytope spanned by points of $S$ and hence $\ASum[1]^p(S) = 0$.
	So assume that $p$ is not an extreme point of $S$. %
	If $p$ lies close enough to a convex hull facet $f$ of $S$, then $p$ is contained in exactly all polytopes that are incident to~$f$ and empty in $S\setminus\{p\}$.
	Hence, $\ASum[1]^p(S) = \ASum[0](S\setminus\{p\};f) = 1$ by Lemma~\ref{facetlemma}. %
	Otherwise, if $p$ is located arbitrarily, consider a continuous path from $p$ to a position close enough to a convex hull facet of $S$ and move $p$ along this path.
	The path can be chosen such that it avoids all lower dimensional elements in the hyperplane arrangement spanned by $S\setminus\{p\}$ and lies inside the convex hull of~$S$. 
	During this movement, $\ASum[1]^p(S)$ can only change when $p$ crosses a facet~$f$ spanned by $d$ points of $S$. 
	Further, changes can only occur from changing amounts of polytopes that have $f$ as a facet. Similar to $\Xk(S;f)$, let $\Xk(S;f^-)$ be the number of empty convex $k$-vertex polytopes that have $f$ as a boundary facet and lie on the negative side of $f$. 
	When $p$ is moved through $f$ (from its positive to its negative side), 
	then the alternating sum of polytopes that $p$ ``stops being inside'' is $\ASum[0](S\setminus\{p\};f) = 1$, 
	and the alternating sum of polytopes that $p$ ``starts being inside'' is $\ASum[0](S\setminus\{p\};f^-) = 1$ (note that $f$ is not a convex hull facet of~$S$). 
	Hence, the value $\ASum[1]^p(S)$ is the same for all possible positions of $p$ on the path, including the final position for which we already showed $\ASum[1]^p(S) = 1$. 
\end{proof}

Likewise, Theorem~\ref{thm:one_point} generalizes to $\R^d$, with an analogous proof as in Section~\ref{sec:alt_sums}.

\begin{theorem}\label{thm:one_point_rd}
	Given a set $S$ of $n\geq d+1$ points in general position in $\R^d$, $h$ of them extreme, it holds that $\ASum[1](S) = n-h$.
\end{theorem}

However, the results for $\ASum[1]^p(S;e)$ and $\ASum[1](S;e)$ do not carry over, 
as the cyclic ordering of the remaining points around $p$ in the proof of Lemma~\ref{one_fixed_edge_one_fixed_point} 
does not have a direct higher dimensional counterpart.

\subsection{Weighted sums}\label{subsec:weighted_sums_rd}

We next consider higher dimensional versions of the results from Section~\ref{sec:weighted_sums}. 
In the following, we denote by $\Xkl(S;f)$ the number of convex $k$-vertex polytopes with $\ell$ interior points, with vertices and interior points a subset of $S$, that have $f$ as a boundary facet and lie to the positive side of $f$. We denote
by $\Xkl(S;p)$ the number of convex $k$-vertex polytopes with $\ell$ interior points that have $p$ on their boundary.
The notation is again generalized to more required boundary elements by listing them all after $S$.
For example, $\Xkl(S;f,p_1,p_2,p_3)$ denotes the number of convex $k$-vertex polytopes with $\ell$ interior points that have $f$, $p_1$, $p_2$, and $p_3$ on their boundary.

Similar to Lemma~\ref{lem:one_fixed_point}, Theorem~\ref{thm:fkl} can directly be generalized to higher dimensions.

\begin{theorem}\label{thm:fkl_rd}
	For any function $\fkl{k}{\ell}$ that fulfills the equation
	\begin{equation} \label{equ:fkl_rd}
		\fkl{k}{\ell} = \fkl{k+1}{\ell-1} + \fkl{k}{\ell-1},
	\end{equation}
	the sum $\FSum(S)= \sum_{\ell\geq 0} \sum_{k\geq 0} \fkl{k}{\ell} \Xkl(S)$
	is invariant over all sets $S$ of $n\geq d+1$ points in general position in $\R^d$, %
	that is, $\FSum(S)$ only depends on the cardinality of $S$ and the dimension $d$.
\end{theorem}
\begin{proof}
	Consider a point set $S$ in general position, an arbitrary point $p\in S$, and the numbers $\Xkl^p(S)$ of polytopes in $S$ that contain $p$ in the interior.
	When continuously moving the points of $S$, the values $\Xkl(S)$ change exactly when a point $p$ crosses a facet $f$ spanned by points of~$S$, from the positive to the negative side, or vice versa. Consider such a change and denote the resulting point set by $S'$. 
	Assume without loss of generality that $p$ is to the positive side of $f$ in~$S$ and to the negative side of $f$ in $S'$.
	Note that in $S$, all polytopes with $f$ on their boundary and lying on the positive side of $f$ have $p$ in their interior 
	(except for the empty simplex containing $p$ and $f$, which exists before and after the change). 
	Further, for every $\ell\geq 1$ and every polytope (with facets $\{f,f_1,\ldots,f_r\}$) counted in $\Xkl(S;f)$, 
	we have exactly one polytope that is counted in $\Xkl[k,\ell-1](S';f)$ (informally, the same polytope, just without $p$ in its interior), 
	and one polytope counted in $\Xkl[k,\ell-1](S';f_1,\ldots,f_{r},p)$ (informally, the same polytope, just with $p$ added on the boundary outside $f$).
	Symmetrically, for every polytope in $S'$ that has $f$ on its boundary and $p$ in the interior, 
	there are two polytopes in $S$ with one point less in the interior and zero and one point more, respectively, on the boundary.
	As $\fkl{k}{\ell} = \fkl{k+1}{\ell-1} + \fkl{k}{\ell-1}$, this implies that no such point move changes the sum $\FSum(S)$.
\end{proof}

The proofs of the following formulae all go along the same lines as their planar counterparts and are omitted.

\begin{corollary}\label{cor:is_rd}
	For any set $S$ of $n\geq d+1$ points in general position in $\R^d$, it holds that  
	\[  \sum_{k=d+1}^{n} \sum_{\ell=0}^{n-d-1} 2^\ell \Xkl(S) = 2^n - \sum_{k=0}^d \binom{n}{k}. \]
\end{corollary}

\begin{corollary}\label{cor:bms_rd}
	For any set $S$ of $n\geq d+1$ points in general position in $\R^d$ and every integer $d+1 \leq m \leq n$ it holds that  
	\[  \sum_{k=d+1}^{m} \sum_{\ell=m-k}^{n-k} \binom{\ell}{m-k} \Xkl(S) = \binom{n}{m}.\]
\end{corollary}
\begin{corollary}\label{cor:fib1_rd}
For any set $S$ of $n\geq d+1$ points in general position in $\R^d$, it holds that  
	\[\sum_{k=d+1}^{n} \sum_{\ell=0}^{n-d-1} \Fib(k+2\ell) X_{k,\ell}(S) = \Fib(2n)- \sum_{k=0}^{d} \Fib(k) \binom{n}{k}.\]
\end{corollary}

\begin{corollary}\label{cor:fib2_rd}
For any set $S$ of $n\geq d+1$ points in general position in $\R^d$, it holds that  
	\[
		\sum_{k=d+1}^{n} \sum_{\ell=0}^{n-d-1} (-1)^{k+\ell}\Fib(k-\ell) X_{k,\ell}(S) 
		=-\Fib(n) + \sum_{k=0}^{d} (-1)^{k+1}  \Fib(k) \binom{n}{k}.
	\]
\end{corollary}

\begin{corollary}\label{cor_cheby_rd}
	For any set $S$ of $n\geq d+1$ points in general position in $\R^d$, it holds that 
\begin{equation*}\label{equ:ChebyT_rd}
\sum_{k=d\!+\!1}^{n} \sum_{\ell=0}^{n\!-\!d\!-\!1} 2\cos\left(\frac{(2k+\ell)\pi}{3}\right) X_{k,\ell}(S) 
	= 2 \cos\left(\frac{n \pi}{3}\right)  - 2\sum_{k=0}^{d} \binom{n}{k}\cos\left(\frac{2k\pi}{3}\right), \ \mbox{and} 
\end{equation*}
\begin{equation*}\label{equ:ChebyU_rd}
\sum_{k=d\!+\!1}^{n} \sum_{\ell=0}^{n\!-\!d\!-\!1} \frac{2}{\sqrt{3}}\sin\left(\frac{(2k+\ell)\pi}{3} \right) X_{k,\ell}(S) 
	= \frac{2}{\sqrt{3}}\sin\left(\frac{n \pi}{3}\right) 
	  - \frac{2}{\sqrt{3}} \sum_{k=0}^{d} \binom{n}{k}\sin\left(\frac{2k \pi}{3}\right).
\end{equation*}
\end{corollary}

\begin{corollary}\label{all_rd}
For any set $S$ of $n\geq d+1$ points in general position in $\R^d$ and for any $x \in \R$, it holds that
\begin{equation}
\sum_{k=d+1}^n \sum_{\ell=0}^{n-k} x^k \left(1+x\right)^\ell X_{k,\ell} = \left(1+x\right)^n -\sum_{k=0}^{d} x^k \binom{n}{k}.
\label{eq:ps_rd}
\end{equation}
\end{corollary}

From Equation~(\ref{eq:ps_rd}) the following can be deduced, analogous to Proposition~\ref{prop:indy} and Corollary~\ref{cor:last}. 
\begin{proposition}\label{prop:indy_rd}
For any set $S$ of $n$ points in general position in $\R^d$, the maximum number of linearly independent equations 
$F_j(S) =\sum_{k \geq 3} \sum_{\ell \geq 0} f_j(k,\ell) X_{k,\ell}$, where each $f_j(k,\ell)$ satisfies Equation (\ref{equ:fkl_rd}), in terms of the variables $X_{k,\ell}$ is $n-d$.
\end{proposition}

\begin{corollary}\label{cor:last_rd}
Let $S$ be a set of $n\geq d+1$ points in general position in $\R^d$.
Any sum $\sum_{k \geq 3} \sum_{\ell \geq 0} f(k,\ell) X_{k,\ell}$, where the function $f(k,\ell)$ fulfills Equation (\ref{equ:fkl_rd}), can be expressed in terms of $n-d$ equations of the form (\ref{eq:ps_rd}) with distinct values $x \in \R$.
\end{corollary}

\subsection{Moment sums}\label{subsec:moment_sums}

Recall that Theorem~\ref{thm:one_point_rd} states that for any set $S$ of $n$ points in general position in $\R^d$, 
the alternating sum of the numbers of convex polytopes with one interior point is $\ASum[1](S) = n-h$. 
Combining this with the first moment of the numbers of empty convex polytopes $M_1(S) = \sum_{k=0}^d (-1)^{d-k} k\binom{n}{k} + i$ from~\cite{prs_oecp_06}, 
we observe that also in $\R^d$, the difference $M_1(S) - \ASum[1](S) = \sum_{k=0}^d (-1)^{d-k} k\binom{n}{k}$ 
is a function that only depends on the cardinality $n$ of $S$ and hence is independent of the combinatorics of the underlying point set $S$. 
For $r = 2$ and any point set $S$ in general position in $\R^d$ let 
	\[ M_2(S) := \sum_{k\geq d+1} (-1)^{k+1} \frac{k}{r} \binom{k-r-1}{r-1} \Xk(S) = \sum_{k\geq d+1} (-1)^{k+1} \frac{k(k-3)}{2} \Xk(S) \] 
be the \emph{second alternating moment} of $\{\Xk(S)\}_{k\geq d+1}$. 
Then, Theorem~\ref{thm:moment_sums} can be generalized to~$\R^d$ in the following way.

\begin{theorem}\label{thm:moment_sums_rd}
	For any set $S$ of $n\geq d+1$ points in general position in $\R^d$ and for integers $0\leq r\leq 2$  
	it holds that
	\[
		\FrSum(S) := \sum_{k\geq d+1} \ \sum_{\ell=0}^r   (-1)^{k-\ell+1+d} \mrk[r-\ell]{k-\ell} \Xkl(S) %
		= 
			\left\{\begin{array}{ll}
					\sum\limits_{k=0}^d (-1)^{d-k} \binom{n}{k} 	& \mbox{for}\ r=0 \\
					\sum\limits_{k=0}^d (-1)^{d-k} k\binom{n}{k}	& \mbox{for}\ r=1 \\
					\sum_{k=0}^{d} (-1)^{d+k}\ \frac{k(k-3)}{2}\binom{n}{k}	& \mbox{for}\  r=2.\\ 
			\end{array}\right.
	\]
\end{theorem}
\noindent The proof of this theorem is identical to the proof of Theorem~\ref{thm:moment_sums}, except that the summation starts with $k=d+1$ instead of $k=3$. 

%

%

%
%

%
%
%
%
%
%

%
%
%
%
%
%
%
%


\end{document}